\documentclass{article}
\pdfoutput=1
\usepackage{amsmath}
\usepackage{amssymb}
\usepackage{amsthm}
\usepackage{authblk}
\usepackage{fullpage}

\newtheorem*{theorem*}{Theorem}
\newtheorem*{corollary*}{Corollary}
\newtheorem*{lemma*}{Lemma}
\newtheorem*{remark*}{Remark}

\title{Characterizing $(0,\pm 1)$-matrices with only even-rank principal submatrices in terms of skew-symmetry}
\author{Robert Brijder\footnote{R.B.\ is a postdoctoral fellow of the Research Foundation -- Flanders (FWO).}}
\date{}
\affil{Hasselt University, Belgium}

\begin{document}

\maketitle

\begin{abstract}
We show that every principal submatrix of a square matrix $A$ with only entries in $\{0,1,-1\}$ has even rank if and only if $A$ is skew-symmetric up to multiplication of rows/columns by $-1$.
\end{abstract}

We say that an $n \times n$-matrix $A$ over some field $\mathbb{F}$ is \emph{skew-symmetric} if $A^T=-A$ and every diagonal entry is zero (the latter condition is redundant if $\mathbb{F}$ is not of characteristic two). The \emph{principal submatrix} of $A$ induced by $I \subseteq \{1,\ldots,n\}$ is the matrix obtained from $A$ by removing the rows and columns indexed by $\{1,\ldots,n\} \setminus I$. Since the rank of a skew-symmetric matrix is even, the rank is even of every principal submatrix of a matrix $A$ that can be obtained from a skew-symmetric matrix by multiplying rows and columns by nonzero scalars. We show that the converse holds in case $A$ has only entries in $\{0,1,-1\}$.

We first recall the Schur complement. If a square matrix $A$ is of the form
$
\begin{pmatrix}
A_1 & A_2 \cr
A_3 & A_4
\end{pmatrix}
$
with $A_1$ invertible, then the \emph{Schur complement} of $A$ on $A_1$ is the matrix $A/A_1 := A_4 - A_3 A_1^{-1} A_2$. The Guttman rank additivity formula says that $r(A) = r(A_1)+r(A/A_1)$, where $r$ denotes the rank of a matrix (see, e.g., \cite[Section~0.9]{SchurBook2005}).

\begin{lemma*}
Let $A$ be an $n \times n$-matrix, with $n \geq 3$, over some field and of the form
\[
\begin{pmatrix}
0 &      -c &       -1 &        &        &         &         \cr
c &       0 &        0 &     -1 &        &         &         \cr
1 &       0 &   \ddots & \ddots & \ddots &         &         \cr
  &       1 &   \ddots & \ddots & \ddots &      -1 &         \cr
  &         &   \ddots & \ddots & \ddots &       0 &      -1 \cr
  &         &          &      1 &      0 &       0 &       b \cr
  &         &          &        &      1 &       a &       0
\end{pmatrix},
\]
where $a,b,c \in \{-1,1\}$ and the blank upper-right and lower-left regions have only zero entries. In particular, if $n=3$ or $n = 4$, then we have 
\[
A = 
\begin{pmatrix}
0 & -c & -1 \cr
c &  0 &  b \cr
1 &  a &  0
\end{pmatrix}
\quad \text{or} \quad
A = 
\begin{pmatrix}
0 & -c & -1 &  0 \cr
c &  0 &  0 & -1 \cr
1 &  0 &  0 &  b \cr
0 &  1 &  a &  0
\end{pmatrix},
\]
respectively. 

If the rank of $A$ is even, then $a = -b$ (i.e., $A$ is skew-symmetric).
\end{lemma*}

\begin{proof}
First assume that $n \geq 5$. Since 
\[
\begin{pmatrix}
0 & 0 \cr
0 & 0
\end{pmatrix}
-
\begin{pmatrix}
1 & 0 \cr
0 & 1
\end{pmatrix}
\cdot
\begin{pmatrix}
0 & -c \cr
c & 0
\end{pmatrix}^{-1}
\cdot
\begin{pmatrix}
-1 & 0 \cr
0 & -1
\end{pmatrix}
=
\begin{pmatrix}
0 & 1/c \cr
-1/c & 0
\end{pmatrix},
\]
we observe that applying the Schur complement on the $2 \times 2$ principal submatrix 
$
A_1 = \begin{pmatrix}
0 & -c \cr
c & 0
\end{pmatrix}
$
induced by $\{1,2\}$ (i.e., the first two rows/columns) obtains a matrix $A/A_1$ that is of the same form as $A$ above. By the above recalled Guttman rank additivity formula, it suffices to verify the cases where $n = 3$ and $n = 4$. 

In the case where $n = 3$, applying the Schur complement on the principal submatrix 
$
\begin{pmatrix}
0 & -c \cr
c & 0
\end{pmatrix}
$
induced by $\{1,2\}$ obtains the $1 \times 1$-matrix with entry $(a + b)/c$ which is zero if and only if $a = -b$.

In the case where $n = 4$, applying the Schur complement in the same way as above obtains the $2 \times 2$-matrix
$
\begin{pmatrix}
0 & b+1/c \cr
a-1/c & 0
\end{pmatrix}
$
which has rank $0$ if and only if $a = 1/c = -b$ and has rank $2$ if and only if $a \neq 1/c \neq -b$. Since $a, b, c \in \{-1,1\}$, we have $a = -1/c = -b$ in the latter case.
\end{proof}

We are now ready to state the main result.
\begin{theorem*}
Let $A$ be an $n \times n$-matrix over some field with only entries in $\{0,1,-1\}$. Then every principal submatrix of $A$ has even rank if and only if $A$ is skew-symmetric up to multiplication of rows/columns by $-1$.
\end{theorem*}
\begin{proof}
It is well-known that every skew-symmetric matrix has even rank, see, e.g., \cite[Chapter XV, Theorem~8.1]{LangAlgebraBook}. 
Therefore, the if implication holds. 

Conversely, assume that every principal submatrix of $A$ has even rank. Consequently, every diagonal entry of $A$ is zero and entry $A_{i,j} \neq 0$ if and only if $A_{j,i} \neq 0$. For every index $i$, we denote by $m(i)$ the smallest index such that $A_{m(i),i} \neq 0$ (or, equivalently, $A_{i,m(i)} \neq 0$). If there is no such value, then we set $m(i)$ to be equal to $n+1$.

Without loss of generality, we may assume the rows and columns are simultaneously reordered such that if $1 < i \leq j$ for indices $i$ and $j$, then $m(i) \leq m(j)$. As a result, we have that if the row/column with index $i \neq 1$ (i.e., not the first row/column) is nonzero, then $m(i) < i$.

Since $1 < i \leq j$ implies $m(i) \leq m(j)$, by multiplying nonzero rows and columns with index $i > 1$ by $-1$ when $A_{i,m(i)} = -1$ and $A_{m(i),i} = 1$, respectively, in the order of increasing $i$, we may assume that $A_{i,m(i)} = -A_{m(i),i} = 1$ for all indices $i$. 

We prove that $A$ is now skew-symmetric. Let $k>l$ be such that $A_{k,l} \neq 0$. Assume  that for all $(i,j) \neq (k,l)$ with $i \leq k$ and $j \leq l$ we have $A_{i,j} = -A_{j,i}$. We show that $A_{k,l} = -A_{l,k}$. Consider the sequence $S = k,l,m(k),m(l),m^2(k),m^2(l),\ldots$ (here, $m^t$ denotes iteratively applying function $m$ $t$ times). In other words, $S$ is such that element $s_i$ is equal to $k$ if $i = 1$, to $l$ if $i=2$, and to $m(s_{i-2})$ if $i \geq 3$. 
Since $A_{k,l} \neq 0$, we have $m(k) \leq l$ because, by definition, $m(k)$ is the smallest index such that $A_{k,m(k)} \neq 0$. Thus, we have $k \geq l \geq m(k)$. Also, recall that $1 < i \leq j$ implies $m(i) \leq m(j)$. Hence, $k \geq l > 1$ implies $m(k) \geq m(l)$, and $l \geq m(k) > 1$ implies $m(l) \geq m^2(k)$, and so we obtain $k \geq l \geq m(k) \geq m(l) \geq m^2(k)$ assuming $l$ and $m(k)$ are not equal to $1$. By iteration, we observe that, for all $x$, we have $s_x \geq s_{x+1}$ if $s_y > 1$ for all $y \in \{2,\ldots,x-1\}$.

Let $s_1,\ldots,s_q$ be the largest prefix of $S$ that is strictly decreasing. Since $k>l$, we have $q \geq 2$. If $q = 2$, then $l = m(k)$ and so $A_{k,l} = A_{k,m(k)} = -A_{m(k),k} = A_{l,k}$ by the construction. Assume now that $q \geq 3$. Since $s_{q-2} \geq s_{q-1} > 1$, we have $s_q = m(s_{q-2}) \geq m(s_{q-1}) = s_{q+1}$. So, $s_q = s_{q+1}$. The principal submatrix induced by $s_1,\ldots,s_q$ is of the following form.
\[
\bordermatrix{
  & s_{q+1} = s_q & s_{q-1} &   \cdots & \cdots & \cdots & s_2 = l & s_1 = k \cr
s_{q+1} = s_q & 0 &      -1 &       -1 &        &        &         &         \cr
s_{q-1}       & 1 &       0 & -c_{q-3} &     -1 &        &         &         \cr
\vdots        & 1 & c_{q-3} &   \ddots & \ddots & \ddots &         &         \cr
\vdots        &   &       1 &   \ddots & \ddots & \ddots &      -1 &         \cr
\vdots        &   &         &   \ddots & \ddots & \ddots &    -c_1 &      -1 \cr
s_2 = l       &   &         &          &      1 &    c_1 &       0 & A_{l,k} \cr
s_1 = k       &   &         &          &        &      1 & A_{k,l} &       0
}
\]
Indeed, if $i \geq 3$, then $s_i = m(s_{i-2})$ and so $A_{s_i,s_{i-2}} = -1$, $A_{s_{i-2},s_i} = 1$, and the blank upper-right and lower-left regions have only zero entries. If all $c_i$'s in this matrix are zero, then the matrix is of the form of the lemma and so $A_{k,l} = -A_{l,k}$. Otherwise, let $t$ be the smallest value such that $c_t \neq 0$. Then the principal submatrix induced by $s_1,\ldots,s_{t+2}$ is of the form of the lemma and so also in this case we have $A_{k,l} = -A_{l,k}$.
\end{proof}

\begin{remark*}
One may wonder whether or not the result still holds if we are dropping the assumption that all entries are in $\{0,1,-1\}$ and allowing multiplication of rows/columns by arbitrary nonzero scalars. This is not the case. Indeed, every strict principal submatrix of the matrix $A$ of the form of the lemma with $n = 4$, $c = 1$, and allowing $b$ and $c$ to be arbitrary nonzero values is skew-symmetric up to multiplication of rows/columns by nonzero scalars. Moreover, by the proof of the lemma, $A$ is of even rank if and only if $a = 1 = -b$ or $a \neq 1 \neq -b$. Taking, e.g., $a = -1$ and $b \notin \{0,1,-1\}$ (such $b$ exists when the field is not $\mathrm{GF}(2)$ or $\mathrm{GF}(3)$) thus obtains a matrix where every principal submatrix is of even rank. However, it is easy to verify that this matrix is not skew-symmetric up to multiplication of rows/columns by nonzero scalars.
\end{remark*}

\begin{remark*}
In the proof of the theorem it is important to first simultaneously reorder the rows/columns 
before multiplying rows and columns by $-1$ in the way described in the proof. Indeed, 
\[
\begin{pmatrix}
0 &  0 & -1 &  0 \cr
0 &  0 &  0 & -1 \cr
1 &  0 &  0 &  1 \cr
0 &  1 &  1 &  0
\end{pmatrix}
\]
is a matrix $A$ 
such that every entry $A_{i,j}$ with $i = m(j)$ or $j = m(i)$ is equal to $1$ if $j \leq i$ and to $-1$ otherwise. While $A$ is not skew-symmetric, it \emph{can} be transformed into a skew-symmetric matrix by multiplying rows and columns by $-1$ (e.g., multiply the first column and the third row by $-1$). 
\end{remark*}

\begin{corollary*}
Let $A$ be an $n \times n$-matrix over $\mathrm{GF}(3)$. Then every principal submatrix of $A$ has even rank if and only if $A$ is skew-symmetric up to multiplication of rows/columns by $-1$.
\end{corollary*}

\bibliography{mmatroids}

\end{document}